\documentclass{amsart}
\usepackage[foot]{amsaddr}
\usepackage{amsmath}
\usepackage{amsthm}
\usepackage{amsfonts,mathrsfs}
\usepackage{amssymb}
\usepackage{graphicx}
\usepackage{enumitem}
\usepackage{color}
\usepackage{verbatim}
\usepackage[normalem]{ulem}

\theoremstyle{plain}
\newtheorem{theorem}{Theorem}[section]
\newtheorem{lemma}[theorem]{Lemma}

\newtheorem{proposition}[theorem]{Proposition}

\theoremstyle{definition}
\newtheorem{definition}[theorem]{Definition}

\newtheoremstyle{TheoremNum}
	{\topsep}{\topsep}              
  {\itshape}                      
  {}                              
  {\bfseries}                     
  {.}                             
  { }                             
  {\thmname{#1}\thmnote{ \bfseries #3}}
\newtheorem{remark}{Remark}

\newcommand{\F}{\mathbb F}
\newcommand{\K}{\mathbb K}
\newcommand{\Z}{\mathbb Z}

\newcommand{\cG}{\mathcal G}

\newcommand{\cH}{\mathcal H}
\newcommand{\cS}{\mathcal S}
\newcommand{\cM}{\mathcal M}
\newcommand{\cT}{\mathcal T}

\newcommand{\bv}{\mathbf v}

\newcommand{\cC}{\mathcal C}

\newcommand{\Aut}{\mathrm{Aut}}

\newcommand{\Tr}{ \ensuremath{ \mathrm{Tr}}}

\newcommand{\RN}[1]{%
  \textup{\uppercase\expandafter{\romannumeral#1}}%
}
\newcommand{\rn}[1]{%
  \textup{\lowercase\expandafter{\romannumeral#1}}%
}

 \def\zhou#1 {\fbox {\footnote {\ }}\ \footnotetext { From Yue: {\color{red}#1}}}

 \def\chen#1 {\fbox {\footnote {\ }}\ \footnotetext { From Rocco: {\color{blue}#1}}}

\begin{document}
	\title{Generalized Twisted Gabidulin codes}
	\author[G. Lunardon]{Guglielmo Lunardon\textsuperscript{\,1}}
	\address{\textsuperscript{1}Dipartimento di Mathematica e Applicazioni ``R. Caccioppoli", Universit\`{a} degli Studi di Napoli ``Federico \RN{2}", I-80126 Napoli, Italy}
	\email[G. Lunardon]{lunardon@unina.it}
	\author[R. Trombetti]{Rocco Trombetti\textsuperscript{\,1}}
	\email[R. Trombetti]{rtrombet@unina.it}
	\author[Y. Zhou]{Yue Zhou\textsuperscript{\,2,3,$\dagger$}}
	\address{\textsuperscript{2}School of Mathematics and Information Science, Guangzhou University, 510006 Guangzhou, China}
	\address{\textsuperscript{3}Department of Mathematics, National University of Defense Technology, 410073 Changsha, China}
	\address{\textsuperscript{$\dagger$}Corresponding author}
	\email[Y. Zhou]{yue.zhou.ovgu@gmail.com}
	\date{\today}

	\begin{abstract}
	Let $\cC$ be a set of $m$ by $n$ matrices over $\F_q$ such that the rank of $A-B$ is at least $d$ for all distinct $A,B\in \cC$. Suppose that $m\leqslant n$. If $\#\cC= q^{n(m-d+1)}$, then $\cC$ is a maximum rank distance (MRD for short) code. Until 2016, there were only two known constructions of MRD codes for arbitrary $1<d<m-1$. One was found by Delsarte (1978) and Gabidulin (1985) independently, and it was later generalized by Kshevetskiy and Gabidulin (2005). We often call them (generalized) Gabidulin codes. Another family was recently obtained by Sheekey (2016), and its elements are called twisted Gabidulin codes. In the same paper, Sheekey also proposed a generalization of the twisted Gabidulin codes. However the equivalence problem for it is not considered, whence it is not clear whether there exist new MRD codes in this generalization. We call the members of this putative larger family generalized twisted Gabidulin codes. In this paper, we first compute the Delsarte duals and adjoint codes of them, then we completely determine the equivalence between different generalized twisted Gabidulin codes. In particular, it can be proven that, up to equivalence, generalized Gabidulin codes and twisted Gabidulin codes are both proper subsets of this family.
	\end{abstract}
	\keywords{maximum rank distance code; linearized polynomial; semifield}
	\maketitle
\section{Introduction}
Let $\K$ be a field. Clearly, the set $\K^{m\times n}$ of all $m\times n$ matrices over $\K$ is a $\K$-vector space. The \emph{rank metric distance} on the $\K^{m\times n}$ is defined by $d(A,B)=\mathrm{rank}(A-B)$ for $A,B\in \K^{m\times n}$.

A subset $\cC\subseteq \K^{m\times n}$ is called a \emph{rank metric code}. The \emph{minimum distance} of $\cC$ is 
\[d(\cC)=\min_{A,B\in \cC, A\neq B} \{d(A,B)\}.\]
When $\cC$ is a $\K$-linear subspace of $\K^{m\times n}$, we say that $\cC$ is a $\K$-linear code and its dimension $\dim_{\K}(\cC)$ is defined to be the dimension of $\cC$ as a subspace over $\K$. In this paper, we restrict ourselves to the cases $\K=\F_q$, where $\F_q$ denotes a finite field of order $q$.  

Let $\cC\subseteq \F_q^{m\times n}$. When $d(\cC)=d$, it is well-known that 
$$\#\cC\le q^{\max\{m,n\}(\min\{m,n\}-d+1)},$$
which is the Singleton-like bound for the rank metric distance; see \cite{delsarte_bilinear_1978}. When the equality holds, we call $\cC$ a \emph{maximum rank distance} (MRD for short) code. MRD codes have various applications in communications and cryptography; for instance, see \cite{gabidulin_public-key_1995,koetter_coding_2008}. More properties of MRD codes can be found in \cite{delsarte_bilinear_1978,gabidulin_MRD_1985,gadouleau_properties_2006,morrison_equivalence_2014,ravagnani_rank-metric_2015}.

A trivial example $\cC$ of MRD codes in $\F_q^{m\times n}$ with $d(\cC)=m$ (here $m\le n$) can be obtained as follows: Take all elements in $a\in \F_{q^n}$, write the linear maps $x\mapsto ax$ as $n\times n$ matrices $M_a$ over $\F_q$. Then we get an MRD code $\cC=\{LM_a: a\in \F_{q^n}\}$, where $L$ is an $m\times n$ matrix of rank $m$. If we replace $ax$ by $a\circ x$ where $\circ$ is the multiplication of a prequasifield of order $q^n$, then still we can get an MRD code. This code is $\F_q$-linear if and only if the prequasifield is a presemifield which is isotopic to a semifield with kernel containing $\F_q$; see \cite{de_la_cruz_algebraic_2016}.

In \cite{delsarte_bilinear_1978} and \cite{gabidulin_MRD_1985}, Delsarte and Gabidulin independently construct the first family of $\F_q$-linear MRD codes of size $q^{nk}$ over finite fields $\F_q$ for every $k$ and $n$. This family is generalized by Kshevetskiy and Gabidulin in \cite{kshevetskiy_new_2005}. If we restrict ourselves to the MRD codes of $n\times n$ matrices, this family can be defined equivalently in the following manner.
\begin{definition}\label{def:GG}
	Let $n,k,s\in \Z^+$ such that $\gcd(n,s)=1$ and $q$ a power of a prime. Then the set 
	\[\cG_{k,s} = \{a_0 x + a_1 x^{q^{s}} + \dots a_{k-1} x^{q^{s(k-1)}}: a_0,a_1,\dots, a_{k-1}\in \F_{q^n} \}\]
	is an $\F_q$-linear MRD code of size $q^{nk}$, which we call a
	\emph{generalized Gabidulin code}.
\end{definition}

Actually, generalized Gabidulin codes are defined in a different way. Under a given basis of $\F_{q^n}$ over $\F_q$, it is well-known that each element $a$ of $\F_{q^n}$ can be written as a (column) vector $\bv(a)$ in $\F_{q}^n$. Let $\alpha_1,\dots,\alpha_m$ be a set of linearly independent elements of $\F_{q^n}$ over $\F_q$, where $k\le m\le n$. Then
\[\left\{ \left(\bv(f(\alpha_1)), \dots, \bv(f(\alpha_m))\right)^T: f\in \cG_{k,s}
\right\}\]
is the original generalized Gabidulin code consisting of $m\times n$ matrices, where $(\cdot )^T$ denotes the transpose of a matrix. To get the minimum distance of this code, we only have to concentrate on the number of the roots of each $f\in \cG_{k,s}$. More precise, it can be shown that every polynomial in $\cG_{k,s}$ has at most $q^{k-1}$ roots. Also, for any $k-1$ dimensional subspace of the space spanned by the $\alpha_i$'s, there always exists a polynomial in $\cG_{k,s}$ vanishing on it, which implies that the rank of the associated matrix is $m-k+1$.

In $\cG_{k,s}$, we see that all its members are of the form $f(x)= \sum_{i=0}^{n-1}a_i x^{q^i}$, where $a_i\in \F_{q^n}$. A polynomial of this form is called a \emph{linearized polynomial} (also a $q$-polynomial because its exponents are all powers of $q$). They are equivalent to $\F_q$-linear transformations from $\F_{q^n}$ to itself. We refer to \cite{lidl_finite_1997} for their basic properties.

In the rest of this paper, we will consider $n\times n$ MRD codes as subsets of linearized polynomials.

Recently, Sheekey \cite{sheekey_new_2016} made a breakthrough in the construction of new linear MRD codes using linearized polynomials.
\begin{definition}\label{def:Sheekey}
	Let $n,k,h\in \Z^+$ and $k<n$. Let $\eta$ be in $\F_{q^n}$ such that $N_{q^n/q}(\eta)\neq (-1)^{nk}$, where $N_{q^n/q}(\eta)= \eta^{1+q+\cdots+q^{n-1}}$. Then the set 
	\[\cH_k(\eta, h) = \{a_0 x + a_1 x^q + \dots a_{k-1} x^{q^{k-1}} + \eta a_0^{q^h} x^{q^k}: a_0,a_1,\dots, a_{k-1}\in \F_{q^n} \}\]
	is an $\F_q$-linear MRD code of size $q^{nk}$, which is called a \emph{twisted Gabidulin code}.
\end{definition}

Another recent progress is a family of nonlinear MRD codes in $\F_q^{3\times 3}$ with minimum distance $2$ constructed by Cossidente, Marino and Pavese in \cite{cossidente_non-linear_2016}, which was later generalized by Durante and Siciliano \cite{durante_nonlinear_MRD_2017}.

After the submission of this paper, Otal and \"Ozbudak~\cite{otal_additive_2016} proved that the twisted Gabidulin codes can be further generalized into additive MRD codes. Moreover, several new families of MRD codes consisting of $n\times n$ matrices have been constructed, including
\begin{itemize}
	\item the non-additive family constructed by Otal and \"Ozbudak in \cite{otal_non-additive_2018};
	\item linear MRD codes associated with maximum scattered linear sets in $\mathrm{PG}(1,q^6)$ and $\mathrm{PG}(1,q^8)$ presented in \cite{csajbok_newMRD_2017} and \cite{csajbok_maximum_arxiv};
	\item the family of linear ones appeared in \cite{trombetti_new_2017} which is related to Hughes-Kleinfeld semifields.
\end{itemize}

In this paper, we investigate a generalization of the twisted Gabidulin codes, which is mentioned in \cite[Remark 8]{sheekey_new_2016}. We call them \emph{generalized twisted Gabidulin codes}. We can show that, up to equivalence, the generalized Gabidulin codes and twisted Gabidulin codes are both quite small proper subsets of this new family of MRD codes.

The organization of this paper is as follows: In Section \ref{sec:dual} we give a brief introduction of the Delsarte dual codes and the adjoint codes of MRD codes. In Section \ref{sec:generalization}, we present an alternative proof of the fact that generalized twisted Gabidulin codes are MRD codes. Finally, we completely determine the equivalence between their different members.

\section{Dual and adjoint codes of MRD codes}\label{sec:dual}
We define a symmetric bilinear form on the set of $m\times n$ matrices by
\[\langle M,N\rangle:= \Tr(MN^T),\]
where $N^T$ is the transpose of $N$. The \emph{Delsarte dual code} of an $\F_q$-linear code $\cC$ is 
\[\cC^\perp :=\{M\in \F_q^{m\times n}:\langle M,N \rangle=0,\forall N\in \cC  \}.\]

One important result proved by Delsarte \cite{delsarte_bilinear_1978} is that the Delsarte dual code of a linear MRD code is still MRD. As we are considering MRD codes using linearized polynomials, for $m=n$ we give the definition of Delsarte dual for polynomials; see \cite{sheekey_new_2016} too. 

We define the bilinear form $b$ on $q$-polynomials by
\[b\left( f,g \right)=\Tr_{q^n/q}\left(\sum_{i=0}^{n-1}a_ib_i\right),\]
where $f(x)=\sum_{i=0}^{n-1}a_ix^{q^i}$ and $g(x)=\sum_{i=0}^{n-1}b_ix^{q^i}\in \F_{q^n}[x]$. The \emph{Delsarte dual code} $\cC^\perp$ of a set of $q$-polynomials $\cC$ is
\[\cC^\perp=\{f: b(f,g)=0,\forall g\in \cC \} .\]

Let $\cC$ be an MRD code in $\K^{m\times n}$. It is obvious that $\{M^T: M\in \cC\}$ is also an MRD code, because the ranks of $M^T$ and $M$ are the same. When $\K=\F_q$ and $m=n$, we can also interpret the transposes of matrices as an operation on linearized polynomials.

Following the terminology in \cite{sheekey_new_2016}, we define the \emph{adjoint} of a linearized polynomial $f(x)=\sum_{i=0}^{n-1}a_i x^{q^i}$ by $\hat{f}(x):=\sum_{i=0}^{n-1}a_{i}^{q^{n-i}} x^{q^{n-i}}$. If $\cC$ is an MRD code consisting of $q$-polynomials, then the \emph{adjoint code} of $\cC$ is $\widehat{\cC}:=\{\hat{f}: f\in\cC\}$. In fact, the adjoint of $f$ is equivalent to the transpose of the matrix derived from $f$. This result can be found in \cite{kantor_commutative_2003}.

\section{Generalized twisted Gabidulin codes}\label{sec:generalization}

In \cite[Remark 7]{sheekey_new_2016}, a generalization of the twisted Gabidulin codes is presented, and its proof relies on the results by Gow and Quinlan in \cite{gow_galois_2009}. In this section, we present an alternative proof.

\begin{lemma}\label{le:restriction}
	Let $q$ be a prime power, and $\delta,s$ and $n\in \Z^+$ such that $\gcd(n,s)=1$ and $\delta<n$. Let $U$ be an $\F_{q^s}$-subspace of $\F_{q^{sn}}$ and $\dim_{\F_{q^s}}(U)=\delta$. Then
	\begin{equation}\label{eq:restriction_U}
		\dim_{\F_q}(U\cap \F_{q^n})\le \delta.
	\end{equation}
\end{lemma}
\begin{proof}
	Assume that \eqref{eq:restriction_U} is not true, i.e.\ there exist $u_0,u_1,\dots, u_{\delta}\in U\cap \F_{q^n}$ that are  linearly independent over $\F_q$. It implies that
	the $\F_q$-rank of the matrix
	\[M=
	\left(
	  \begin{array}{cccc}
	    u_0 & u_1 & \dots & u_\delta \\
	    u_0^q & u_1^q & \dots & u_\delta^q \\
	    \ldots & \ldots & \ldots & \ldots \\
	    u_0^{q^{n-1}} & u_1^{q^{n-1}} & \dots & u_\delta^{q^{n-1}} \\
	  \end{array}
	\right)
	\]
	is $\delta+1$. Hence there exists at least one $(\delta+1)\times (\delta+1)$ submatrix $N$ of $M$, such that $\det(N)\neq 0$. As $\dim_{\F_{q^s}}(U)=\delta$, the $\F_{q^s}$-rank of the matrix
	\[M'=
	\left(
	  \begin{array}{cccc}
	    u_0 & u_1 & \dots & u_\delta \\
	    u_0^{q^s} & u_1^{q^s} & \dots & u_\delta^{q^s} \\
	    \ldots & \ldots & \ldots & \ldots \\
	    u_0^{q^{s(n-1)}} & u_1^{q^{s(n-1)}} & \dots & u_\delta^{q^{s(n-1)}} \\
	  \end{array}
	\right)
	\]
	is less than or equal to $\delta$. Hence the determinant of any $(\delta+1)\times (\delta+1)$ submatrix of $M'$ is zero.
	
	However, as $\gcd(s,n)=1$, after applying a row permutation on $M$, we get $M'$. It leads to a contradiction on the determinant of $N$.
\end{proof}

\begin{theorem}\label{th:restriction}
	Let $q$ be a prime power, and $s,n\in \Z^+$ such that $\gcd(n,s)=1$. Let $\cM$ be an $\F_{q^s}$-linear MRD code as a set of $q^s$-polynomials over $\F_{q^{sn}}$,  and the dimension of $\cM$ is $k$ where $k<n$. Let $\widetilde{\cM}$ be the intersection of $\cM$ and the set of $q$-polynomials over $\F_{q^n}$. Then  $\widetilde{\cM}$ is an $\F_q$-linear MRD code as a set of $q$-polynomials over $\mathbb{F}_{q^n}$ if and only if $\#\widetilde{\cM}= q^{nk}$.
\end{theorem}
\begin{proof}
	Let $f$ be a nonzero element of $\widetilde{\cM}$. That means $f$ is a $q^s$-linearized polynomial with no more than $q^{s{(k-1)}}$ roots. In other words, the kernel of $f$ is an $\F_{q^s}$-linear space of dimension $\delta$, where $\delta<k$. By Lemma \ref{le:restriction}, we know that the dimension of the kernel of $f$ in $\F_{q^n}$ is also less than or equals to $\delta$. Therefore $\widetilde{\cM}$ is MRD if and only if it is of size $q^{nk}$. 
\end{proof}

Using Theorem \ref{th:restriction}, we can generalize the twisted Gabidulin codes.
\begin{theorem}\label{th:generalized_sheekey}
	Let $n,k,s,h\in \Z^+$ satisfying $\gcd(n,s)=1$ and $k<n$. Let $\eta$ be in $\F_{q^n}$ such that $N_{q^{sn}/q^s}(\eta)\neq (-1)^{nk}$. Then the set 
	\[\cH_{k,s}(\eta, h) = \{a_0 x + a_1 x^{q^s} + \dots +a_{k-1} x^{q^{s(k-1)}} + \eta a_0^{q^h} x^{q^{sk}}: a_0,a_1,\dots, a_{k-1}\in \F_{q^n} \}\]
	is an $\F_q$-linear MRD code of size $q^{nk}$.
\end{theorem}
\begin{proof}
	By Definition \ref{def:Sheekey}, we know that
	\[\overline{\cH}_{k}(\eta,h)=\{a_0 x + a_1 x^{q^s} + \dots + a_{k-1} x^{q^{s(k-1)}} + \eta a_0^{q^h} x^{q^{sk}}: a_0,a_1,\dots, a_{k-1}\in \F_{q^{sn}} \}\] is an MRD code of size $q^{snk}$. Clearly $\cH_{k,s}(\eta, h)\subseteq \overline{\cH}_{k}(\eta, h)$ and $\#\cH_{k,s}(\eta, h)= q^{nk}$. By Theorem \ref{th:restriction}, we complete the proof.
\end{proof}

The MRD codes in Theorem \ref{th:generalized_sheekey} first appeared in \cite[Remark 8]{sheekey_new_2016} and they were denoted by $\cH_k(\mu,h;\sigma)$ which is ${\cH}_{k,s}(\mu,sh)$ with our notation, where $\sigma$ stands for the field automorphism $x\mapsto x^{q^s}$. We call them \emph{generalized twisted Gabidulin codes}.  When $s=1$, $\cH_{k,s}(\eta, h)$ is the twisted Gabidulin code $\cH_{k}(\eta,h)$. When $\eta=0$, $\cH_{k,s}(\eta, h)$ is exactly the generalized Gabidulin code $\cG_{k,s}$.

In particular, when $k=1$, all elements in $\cH_{1,s}(\eta, h)$ are
\[a_0x + \eta a_0^{q^h} x^{q^{s}},\]
for $a_0\in \F_{q^n}$. They actually define the multiplication of a generalized twisted field, which is a presemifield found by Albert \cite{albert_generalized_1961}.

\section{Equivalences of generalized twisted Gabidulin codes}\label{sec:equivalence}
In the literature, there are different definitions of equivalence for rank metric codes; see \cite{de_la_cruz_algebraic_2016,morrison_equivalence_2014}. As we concentrate on MRD codes in the form of linearized polynomials, we use the following definition.
\begin{definition}\label{def:equivalence}
	Let $\cC$ and $\cC'$ be two set of $q$-polynomials over $\F_{q^n}$. They are equivalent if there exist two permutation $q$-polynomials $L_1$, $L_2$ and $\rho\in \Aut(\F_q)$ such that $\cC' =\{ L_1\circ f^\rho \circ L_2(x) : f\in \cC \}$, where $(\sum a_{i}x^{q^i})^\rho:= \sum a_{i}^\rho x^{q^i}$. The \emph{automorphism group} of $\cC$ consists of all $(L_1,L_2,\rho)$ fixing $\cC$.
\end{definition}

It is well-known and also not difficult to show directly that two MRD codes are equivalent if and only if their duals are equivalent.

Let us look at the dual code of $\cH_{k,s}(\eta, h)$. It follows readily that
\[\cH_{k,s}(\eta, h)^\perp =\left\{b_0 x-\frac{1}{\eta} b_0^{q^h}x^{q^{ks}}+\sum_{i=k+1}^{n-1} b_i x^{q^{is}} : b_i\in \F_{q^n} \right\}. \]
By applying $x\mapsto x^{q^{(n-k)s}}$ to the elements of $\cH_{k,s}(\eta, h)^\perp$, we get
\[\left\{-\frac{1}{\eta^{q^{n-ks}}} b_0^{q^{h-ks}}x+\sum_{i=1}^{n-k-1} a_i x^{q^{is}} + b_0^{q^{(n-k)s}}x^{q^{(n-k)s}} : b_0,a_i\in \F_{q^n} \right\}. \]

Multiplying them by some constant and by certain change of variables, we get the following result.
\begin{proposition}\label{prop:dual}
	The Delsarte dual code of $\cH_{k,s}(\eta, h)$ is equivalent to the code $\cH_{n-k,s}(-\eta, n-h)$.
\end{proposition}
Similarly, it is straightforward to prove the following result about the adjoint code of a generalized twisted Gabidulin code.
\begin{proposition}\label{prop:adjoint}
	The adjoint code of $\cH_{k,s}(\eta,h)$ is equivalent to $\cH_{k,s}(1/\eta, sk-h)$.
\end{proposition}

In \cite{sheekey_new_2016}, Sheekey proved that $\cG_{k,s}$ and $\cH_{k,1}(\eta,h)$ are equivalent if and only if $k\in \{1,n-1\}$ and $h\in \{0,1\}$. The equivalence between $\cH_{k,1}(\eta,h)$ and $\cH_{k,1}(\nu,j)$, and the automorphism groups of $\cG_{k,1}$ and $\cH_{k,1}(\eta,h)$ are also completely determined except for the only case $n=4$ and $k=2$. However, for arbitrary parameters $s$, $t$, $g$, $h$, $\eta$ and $\theta$, the equivalence between $\cH_{k,s}(\eta, g)$ and $\cH_{k,t}(\theta, h)$ is still open. In this section, we aim to give a complete answer to this question.
\begin{theorem}\label{th:equivalence_all}
	Let $n,k,s,t,g,h\in \Z^+$ satisfying $\gcd(n,s)=\gcd(n,t)=1$ and $2\le k\le n-2$. Let $\eta$ and $\theta$ be in $\F_{q^n}$ satisfying $N_{q^{sn}/q^s}(\eta)\neq (-1)^{nk}$ and $N_{q^{tn}/q^t}(\theta)\neq (-1)^{nk}$. The codes $\cH_{k,s}(\eta, g)$ and $\cH_{k,t}(\theta, h)$ are equivalent if and only if one of the following collections of conditions are satisfied:
	\begin{enumerate}[label=(\alph*)]
		\item $s\equiv t \pmod{n}$, $g \equiv h \pmod{n}$ and there exist $c,d\in \F_{q^n}^*$, $\rho\in \Aut(\F_q)$ and an integer $r$ such that 
		$$\theta c^{q^h-1} d^{q^{r+h} - q^{r+sk}} = \eta^{\rho q^r}.$$
		\item $s\equiv -t \pmod{n}$, $g \equiv -h \pmod{n}$ and there exist $c,d\in \F_{q^n}^*$, $\rho\in \Aut(\F_q)$ and an integer $r$ such that
		\[c^{q^{g}-1} d^{q^{r+g} - q^{r+sk}} = \eta^{\rho q^r}\theta^{q^{sk}}.\]
	\end{enumerate}
\end{theorem}

Theorem \ref{th:equivalence_all} does not cover the case $k=1$, in which $\cH_{k,s}(\eta, g)$ is exactly a generalized twisted field and the equivalence between $\cH_{1,s}(\eta, g)$ and $\cH_{1,t}(\theta, h)$ is exactly the isotopism between two generalized twisted fields. This problem is completely solved; see \cite{albert_isotopy_1961,biliotti_collineation_1999}. For $k=n-1$, we can convert the equivalence between $\cH_{n-1,s}(\eta, g)$ and $\cH_{n-1,t}(\theta, h)$ into the equivalence between their Delsarte duals.

Before considering the equivalence between distinct members in the generalized twisted Gabidulin codes, namely, the proof of Theorem \ref{th:equivalence_all}, we introduce some results and tools.

In the rest of this section, all calculations of integers and indices are taken modulo $n$, because we are essentially considering the terms $x^{q^i}$ of polynomials in $\F_{q^n}[x]/ (x^{q^n}-x)$.

\begin{lemma}\label{le:independent_support}
	Let $L_1, L_2$ be two $q$-polynomials over $\F_{q^n}$ and $\rho\in\Aut(\F_q)$. For $k \in \{0,\dots, n-1\}$, we define $\cM_k = \{ax^{q^k}: a\in\F_{q^n} \}$ and 
	$$I_k= \{i: \text{the }q^i\text{-th coefficient of } L_1 \circ g^\rho \circ L_2 \text{ is not zero for some }g\in \cM_k  \}.$$ 
	Then 
	\[I_{\bar{k}} = I_{k}+\bar{k}-k= \{i+\bar{k}-k \pmod{n}: i\in I_k\}\] 
	for any $k,\bar{k}\in \{0,\dots, n-1\}$.
\end{lemma}
\begin{proof}
	Assume that $L_1(x):= \sum_{i=0}^{n-1}c_i x^{q^i}$ and $L_2(x):= \sum_{i=0}^{n-1}d_i x^{q^i}$, where $c_i$, $d_i\in \F_{q^n}$.  For any $0\leqslant k \leqslant n-1$ and $b\in \F_{q^n}$, let $g(x)= bx^{q^k}$. Recall that $g^\rho(x) := b^\rho x^{q^k}$. By calculation, 
	\begin{align*}
	L_1\circ g^\rho  \circ L_2=& \sum_{j=0}^{n-1} c_j \left(\sum_{i=0}^{n-1} b^\rho d^{q^k}_i x^{q^{i+k} } \right)^{q^j} \\
	=& \sum_{l=0}^{n-1}\left(\sum_{j=0}^{n-1} c_j(b^\rho d^{q^k}_{l-j-k})^{q^j}\right)x^{q^l}\\
	=&\sum_{l=0}^{n-1}\left(\sum_{j=0}^{n-1} (c_j d_{l-j-k}^{q^{j+k}}) b^{\rho q^j}\right)x^{q^l}.
	\end{align*}
	We define $f_{l,k}(y) =\sum_{j=0}^{n-1} (c_j d_{l-j-k}^{q^{j+k}}) y^{q^j}$.	Hence, under $(L_1, L_2,\rho)$, $bx^{q^k}$ is mapped to $\sum_{l=0}^{n-1}f_{l,k}(b^\rho)x^{q^l}$.
	
	By the definition of $I_k$, the polynomial $f_{l,k}$ is constantly zero for each $l\notin I_k\subseteq\{0,1,\dots, n-1\}$. It is clear that the value of $f_{l,k}(y)$ is zero for all $y\in \F_{q^n}$ if and only if $c_j d_{l-j-k}^{q^{j+k}}=0$, i.e., $c_j d_{l-j-k}=0$  for all $j$. Therefore, for any $\bar{k}\in \{0,1,\dots, n-1\}$, $f_{l,\bar{k}}$ is constantly zero if and only if $l\notin I_k+\bar{k}-k$. 
\end{proof}

The observation in Lemma \ref{le:independent_support} shows us that to investigate the equivalence problem for two given sets $\cC_1$ and $\cC_2$ of $q$-polynomials viewed as rank metric codes, sometimes we only have to concentrate on the $q$-degrees of the nonzero terms of the $q$-polynomials appearing in $\cC_1$ and $\cC_2$.

Since we are going to handle the subsets of $q$-degrees such as $I_k$ in Lemma \ref{le:independent_support} in many of the following proofs, here we give them a name.
Let $\cC$ be a set of $q$-polynomial over $\F_{q^n}$. The \emph{universal support} $\cS(\cC)$ of $\cC$ is defined to be the following subset of $\{0,1,\dots, n-1\}$,
\[\cS(\cC) := \{ i: \text{there exists }f\in \cC \text{ such that the }q^i\text{-coefficient of $f$ is not zero} \}.\]
Assume that there exist a subset $B\subseteq\{0,1,\dots, n-1\}$ and a set of permutations $\{h_i: i \in B\}$ on $\F_{q^n}$ such that
\[\left\{\sum_{i \in B}h_i(a) x^{q^{i}} : a\in \F_{q^n}\right\}\subseteq \cC.\]
Then we call $B$ an \emph{independent support} of $\cC$.  For example, $\{0,1\}$ and $\{1\}$ are both independent supports of $\{a_0 x + (a_0^q + a_1) x^q: a_0,a_1 \in \F_{q^n}\}$.

Let $A$ and $B$ be two subsets of $\{0,1,\dots, n-1\}$. We define another subset $A^B$ of $\{0,1,\dots, n-1\}$ by
\[A^B := \{k: \text{there exists a unique pair } (i,j)\in A\times B \text{ such that }k\equiv i+j \pmod{n}\}.\]
For instance, let $n=4$, then $\{1,2\}^{\{0,1\}}= \{1,3\}$. 

Let $\cC_1$ and $\cC_2$ be two set of $q$-polynomials. Assume that $\cC_1$ and $\cC_2$ are equivalent, whence there exist permutation $q$-polynomials $L_1$, $L_2$ and $\rho\in \Aut(\F_q)$ such that $\cC_1 =\{ L_1\circ f^\rho \circ L_2(x) : f\in \cC_2 \}$. We use $\tau$ to denote this map from $\cC_1$ to $\cC_2$, i.e.\ $\tau(f(x)) := L_1\circ f^\rho \circ L_2(x)$. In Lemma \ref{le:independent_support}, the subset $I_k$ is actually $\cS(\tau(\cM_k))$. Moreover, we can prove the following result.

\begin{lemma}\label{le:equivalence}
	Let $\cC_1$ and $\cC_2$ be two set of $q$-polynomials.  Assume that $\cC_1$ and $\cC_2$ are equivalent by $\tau$. Let $A$ be the universal support of $\{\tau(ax) : a\in\F_{q^n}\}$. Then
	\begin{equation}\label{eq:equivalence_lemma}
		A^B \subseteq \cS(\cC_2),
	\end{equation}
	for every independent support $B$ of $\cC_1$.
\end{lemma}

\begin{proof}
	Let $B$ be an arbitrary independent support of $\cC_1$. By definition, there exists permutations $h_i$ with $i\in B$ such that
	\[\mathcal{B} = \left\{\sum_{i \in B}h_i(a) x^{q^{i}} : a\in \F_{q^n}\right\}\subseteq \cC_1.\]
	
	By Lemma \ref{le:independent_support}, the universal support of $\{\tau(h_i(a) x^{q^{i}}) : a\in \F_{q^n}\}$ is exactly $A+i$. Now we want to derive some information for the universal support $\cS(\tau(\mathcal{B}))$. In fact, for any $l\in \{0,1,\cdots, n-1\}$, if there exists exactly one $i\in B$ such that $l\in A+i$, then we can guarantee that $l\in \cS(\tau(\mathcal{B}))$. This is actually equivalent to
	\[A^B\subseteq \cS(\tau(\mathcal{B})).\]
	By definition, $\cS(\tau(\mathcal{B}))$ must a subset of $\cS(\cC_2)$, whence \eqref{eq:equivalence_lemma} holds.
\end{proof}

Similarly to the proof of Lemma \ref{le:independent_support}, we can show the following result.
\begin{lemma}\label{le:monomial_equivalence}
	Let $L_1$, $L_2$ be two permutation $q$-polynomials over $\F_{q^n}$ and $\rho \in \Aut(\F_q)$. If
	\[\{bx: b\in\F_{q^n} \} = \{ L_1(a^\rho L_2(x)): a\in\F_{q^n} \},\]
	then $L_1(x)=cx^{q^r}$ and $L_2(x)=dx^{q^{n-r}}$ for some integer $r$ and for $c,d\in \mathbb{F}_{q^n}^*$.
\end{lemma}
\begin{proof}
	Let $L_1(x)=\sum_{i=0}^{n-1}c_i x^{q^i}$ and $L_2(x)=\sum_{i=0}^{n-1}d_i x^{q^i}$, where $c_i$ and $d_i\in \F_{q^n}$. Then
	\begin{align*}
		L_1(a^\rho L_2(x))=& \sum_{j=0}^{n-1}c_j \left(\sum_{i=0}^{n-1} a^\rho d_i x^{q^i}\right)^{q^j} \\
		=& \sum_{l=0}^{n-1}\left(\sum_{j=0}^{n-1} c_j(a^\rho d_{l-j})^{q^j}\right)x^{q^l}.
	\end{align*}
	It implies that, when $l\neq 0$, the coefficient of $x^{q^l}$ in $L_1(a^\rho L_2(x))$ always equals $0$, i.e.
	\[\sum_{j=0}^{n-1} c_jd_{l-j}^{q^j}a^{\rho q^j}=0,\]
	for all $a\in \F_{q^n}$. That means $c_jd_{l-j}=0$ for all $j$ and $l\neq 0$. Therefore, there exists a unique integer $r$ between $0$ and $n-1$ such that $c_{r}$ and $d_{n-r}$ are nonzero and all remaining $c_i$ and $d_i$ are zero.
\end{proof}

The proof of the necessary part of Theorem \ref{th:equivalence_all} consists of the results in Theorems \ref{th:main_inequivalence_1}, \ref{th:main_inequivalence_2} and \ref{th:main_inequivalence_3}. First, for $s\not\equiv \pm t \pmod{n}$, we exclude many possibilities for the value of $k$.
\begin{theorem}\label{th:main_inequivalence_1}
	Let $n,k,s,t,g,h\in \Z^+$ satisfying $\gcd(n,s)=\gcd(n,t)=1$. Let $\eta$ and $\theta$ be in $\F_{q^n}$ satisfying $N_{q^{sn}/q^s}(\eta)\neq (-1)^{nk}$ and $N_{q^{tn}/q^t}(\theta)\neq (-1)^{nk}$. Assume that $s\not\equiv \pm t \pmod{n}$.
	\begin{enumerate}[label=(\alph*)]
		\item If $\eta=\theta=0$ and $1<k<n-1$, then $\cH_{k,s}(\eta, g)$ and $\cH_{k,t}(\theta, h)$ are not equivalent.
		\item If $\eta=0$, $\theta\neq 0$ and $2<k<n-2$, then $\cH_{k,s}(\eta, g)$ and $\cH_{k,t}(\theta, h)$ are not equivalent.
		\item If both $\eta$ and $\theta$ are nonzero and $2<k<n-2$, then $\cH_{k,s}(\eta, g)$ and $\cH_{k,t}(\theta, h)$ are not equivalent.
	\end{enumerate}
\end{theorem}
\begin{proof}
	According to Proposition \ref{prop:dual}, the Delsarte dual of $\cH_{n-k,s}(\eta, g)$ is equivalent to ${\cH}_{k,s}(-\eta,n-g)$.	As two MRD codes are equivalent if and only if their Delsarte duals are equivalent, we can assume that $ k\le \lfloor\frac{n}{2} \rfloor$. As $\gcd(n,t)=1$,  there is an integer $l$ such that $s\equiv lt \pmod{n}$ and $1<l<n-1$.
	
	Assume that $\cH_{k,s}(\eta, g)$ and $\cH_{k,t}(\theta, h)$ are equivalent. We use $\tau$ to denote the map from $\cH_{k,s}(\eta, g)$ to $\cH_{k,t}(\theta, h)$, i.e.\ there exist linearized permutation polynomials $L_1(x)$, $L_2(x)$ and $\rho\in \Aut(\F_q)$ such that $\tau(f(x)) := L_1\circ f^\rho \circ L_2(x)\in \cH_{k,t}(\theta, h)$ for each $f\in \cH_{k,s}(\eta, g)$. As in the proof of Lemma \ref{le:equivalence}, let $A$ be the universal support of $\{\tau(ax) : a\in\F_{q^n}\}$.
	
	\vspace*{2mm}
	\noindent$(a)$ First we look at the case $\eta=\theta=0$, which means that $\cH_{k,s}(\eta, g)$ and $\cH_{k,t}(\theta, h)$ are both generalized Gabidulin codes.
	
	In this case, we take 
	\[\cT_s:=\{\{is\}: i=0,\dots, k-1\},\] 
	which is a set of independent supports of $\cH_{k,s}(0,g)$ and \[\cS(\cH_{k,t}(0, h))= \{it: i=0,\dots, k-1\}.\]
	If $j\in A$, then $j+is\in A^{\{is\}}$ for each $i\in \{0,1,\ldots,k-1\}$. Together with Lemma \ref{le:equivalence}, we have
	\begin{equation}\label{eq:long_eta=theta=0}
		\{j+ is: i=0,\dots, k-1 \}=\bigcup_{B\in \cT_s} \{j\}^B \subseteq \bigcup_{B\in \cT_s} A^B \subseteq \cS(\cH_{k,t}(\theta, h)).
	\end{equation}
	Hence
	\[\{j+ is: i=0,\dots, k-1 \} \subseteq \{it: i=0,\dots, k-1\}.\]
	Letting $j\equiv ut \pmod{n}$ and plugging $s\equiv lt\pmod{n}$ into the above equation, we have
	\begin{equation}\label{eq:short_eta=theta=0}
		\{u+il: i=0,\dots,k-1\}=\{i: i=0, \dots, k-1\}.
	\end{equation}
	
	According to our assumption, we know that $1<l<n-1$. If $l\le \lfloor \frac{n}{2}\rfloor$, then we consider the sequence of integers $u$, $u+l$, $u+2l,\dots$. As both $l$ and $k$ are less than or equal to  $\lfloor \frac{n}{2}\rfloor$, $l\ge 2$ and $u\in \{0,1,\dots, k-1\}$, there must exist an $\alpha\in \{1,\dots, k-1\}$ such that $u+\alpha l \ge k$ and $u+(\alpha -1)l\le k-1$. It is a contradiction to \eqref{eq:short_eta=theta=0}.
	
	If $l\ge \lfloor \frac{n}{2}\rfloor$, then we consider the sequence $u+(k-1)l$, $u+(k-2)l,\dots$, which can be viewed as $u+(k-1)l$, $(u+(k-1)l) + (n-l)$, $(u+(k-1)l) + 2(n-l), \dots$. As $n-l \le \lfloor \frac{n}{2}\rfloor$, similarly to the $l\le \lfloor \frac{n}{2}\rfloor$ case, we can also get a contradiction to \eqref{eq:short_eta=theta=0}.
	
	\vspace*{2mm}
	\noindent$(b)$ Now we consider the case in which $\eta=0$ and $\theta\neq 0$. We take 
	\[\cT_s:=\{\{is\}: i=0,\dots, k-1\},\] 
	which consists of independent supports of $\cH_{k,s}(0,g)$ and \[\cS(\cH_{k,t}(\theta, h))= \{it: i=0,\dots, k\}.\]
	Similarly as in the proof of $(a)$, if $j\in A$, then we have \eqref{eq:long_eta=theta=0} which implies that
	\[\{j+ is: i=0,\dots, k-1 \} \subseteq \{it: i=0,\dots, k\}.\]
	Letting $j\equiv ut \pmod{n}$ and plugging $s\equiv lt\pmod{n}$ into the above equation, we have
	\begin{equation}\label{eq:short_eta=0}
		\{u+il: i=0,\dots,k-1\}\subseteq\{i: i=0, \dots, k\}.
	\end{equation}
	
	As in the proof of $(a)$, we first consider the case $2\le l\le \lfloor \frac{n}{2}\rfloor$. We look at the sequences of integers $u, u+l,\dots,u+(k-1)l$. There must exist an $\alpha$ such that $u+\alpha l > k$ and $u+(\alpha -1)l\le k$, otherwise $u+(k-1)l\le k$ which means $k\le 2$, $l=2$ and $u=0$. It contradicts our assumption that $k> 2$. Hence \eqref{eq:short_eta=0} does not hold. 
	
	For the case $\lfloor \frac{n}{2}\rfloor <l \le n-2$, following the approach in the proof of $(a)$, we can show that \eqref{eq:short_eta=0} is not satisfied.
	
	\vspace*{2mm}
	\noindent$(c)$ Finally,  we consider the case in which both $\eta$ and $\theta$ are nonzero. We take 
	\[\cT_s:=\{\{is\}: i=1,\dots, k-1\},\] 
	which consists of independent supports of $\cH_{k,s}(\eta,g)$ and \[\cS(\cH_{k,t}(\theta, h))= \{it: i=0,\dots, k\}.\]
	Similarly as in the proof of $(a)$, if $j\in A$, then we have 
	\begin{equation}\label{eq:long_eta_theta_nonzero}
		\{j+ is: i=1,\dots, k-1 \}=\bigcup_{B\in \cT_s} \{j\}^B \subseteq \bigcup_{B\in \cT_s} A^B \subseteq \cS(\cH_{k,t}(\theta, h)).
	\end{equation}
	which implies that
	\[\{j+ is: i=1,\dots, k-1 \} \subseteq \{it: i=0,\dots, k\}.\]
	Letting $j\equiv ut \pmod{n}$ and plugging $s\equiv lt\pmod{n}$ into the above equation, we have
	\begin{equation}\label{eq:short_eta_theta_nonzero}
		\{u+il: i=1,\dots,k-1\}\subseteq\{i: i=0, \dots, k\}.
	\end{equation}

	As in the proof of $(a)$, we first consider the case  $2\le l\le \lfloor \frac{n}{2}\rfloor$. From $k,l\le \lfloor\frac{n}{2}\rfloor$ and \eqref{eq:short_eta_theta_nonzero}, we derive that $0\le u+l<\dots<u+(k-1)l\le k$. Noting that there are $k-1$ and $k+1$ integers in $\{u+il: i=1,\dots,k-1\}$ and $\{i: i=0, \dots, k\}$ respectively,  we obtain the following three possible values for $u$ and $k$:
	\begin{enumerate}[label=(\Roman*)]
		\item $u+l=0$ and $k\le\frac{l-u}{l-1} = \frac{2l}{l-1}$, i.e.\
		\[k\le \left\{
		  \begin{array}{ll}
		    4, & \hbox{$l=2$;} \\
		    3, & \hbox{$l=3$;} \\
		    2, & \hbox{$l\ge 4$.}
		  \end{array}
			\right. \]
		\item $u+l=1$ and $k\le \frac{l-u}{l-1} = \frac{2l-1}{l-1}$, i.e.\ \[ k\le \left\{
		\begin{array}{ll}
		    3, & \hbox{$l=2$;} \\
		    2, & \hbox{$l\ge 3$.}
		  \end{array}\right.\]
		\item $u+l=2$ and $k\le \frac{l-u}{l-1} = \frac{2l-2}{l-1}= 2.$
	\end{enumerate}
	According to our assumption on the value of $k$, the case $k=2$ is excluded. 
	
	We take another independent support $B=\{0, ks\}$ of $\cH_{k,s}(\eta,g)$. Next we consider the cases $k=3$ and $k=4$ separately:
	
	\vspace*{2mm}
	\noindent\textbf{Case 1:} Now $k=4$ and $n\ge 2k=8$. From (\RN{1}) we derive that $l=2$, $u=-2$, $j\equiv -2t \pmod{n}$ and $s\equiv 2t \pmod{n}$. That means $A=\{ -2t \}$. By Lemma \ref{le:equivalence},
	\[A^B=\{-2t,6t\}\subseteq \cS =\{0,t,2t,3t,4t\}.\]
	Assume that $i\in \{0,1,\dots, 4\}$ is such that  $-2t\equiv i t \pmod{n}$. It means that $n \mid i+2$ which is less than or equal to $6$. It is a contradiction to $n\ge 8$.
	
	\vspace*{2mm}
	\noindent\textbf{Case 2:} Now $k=3$ and $n\ge 2k =6$. From (\RN{1}) and (\RN{2}), we derive that there are two possible cases:
	\begin{enumerate}[label=(\roman*)]
		\item $l=3$, $u=-3$, $j\equiv -3t \pmod{n}$, $s\equiv 3t \pmod{n}$ and $A=\{-3t\}$;
		\item $l=2$, $u\in \{-l, -l+1\}=\{-2,-1\}$, $j\in \{-2t,-t\}$, $s\equiv 2t \pmod{n}$ and $A\subseteq \{-2t, -t\}$.
	\end{enumerate}
	In case (\rn{1}), we get that
	\[A^B=\{-3t,6t  \}\subseteq \cS=\{0,t,2t,3t\}.\]
	As $n\ge 6$ and $\gcd(n,t)=1$, the only possible value of $-3t$ is $ -3t\equiv 3t \pmod{n}$ which means that $n=6$ and $\gcd(s,n)=\gcd(3t,n)=3$. It is a contradiction.
	
	In case (\rn{2}), we consider three possibilities:
	\begin{itemize}
		\item $A=\{-2t\}$, $A^B=\{-2t,4t  \}\subseteq \cS=\{0,t,2t,3t\}$ which contradicts $n\ge 6$.
		\item $A=\{-t\}$, $A^B=\{-t,5t  \}\subseteq \cS=\{0,t,2t,3t\}$ which again contradicts $n\ge 6$.
		\item $A=\{-t,-2t\}$, $A^B=\{-t,4t  \}\subseteq \cS=\{0,t,2t,3t\}$ and $-2t \equiv 5t \pmod{n}$. That means $n=7$ and $-t\notin \cS$, which contradicts $-t\in A^B\subseteq \cS$.
	\end{itemize}
	Therefore \eqref{eq:short_eta_theta_nonzero} cannot hold, i.e.\ $\cH_{k,s}(\eta, g)$ and $\cH_{k,t}(\theta, h)$ are not equivalent.
	 
	For the case  $\lfloor \frac{n}{2}\rfloor <l \le n-2$, following the approach in the proof of $(a)$ and looking at the sequence $u+(k-1)l,u+(k-2)l,\dots, u+1$, we can again show that \eqref{eq:short_eta_theta_nonzero} does not hold.
\end{proof}

Next, we proceed to investigate the cases in which $k=2$ or $n-2$, and we can prove the following result.

\begin{theorem}\label{th:main_inequivalence_2}
	Let $n,k,s,t,g,h\in \Z^+$ satisfying $\gcd(n,s)=\gcd(n,t)=1$ and $k=2$ or $n-2$. Let $\eta$ and $\theta$ be in $\F_{q^n}$ satisfying $N_{q^{sn}/q^s}(\eta)\neq (-1)^{nk}$ and $N_{q^{tn}/q^t}(\theta)\neq (-1)^{nk}$. Assume that $s\not \equiv \pm t\pmod{n}$. If
	\begin{itemize}
		\item at most one of $\eta$ and $\theta$ is nonzero, or
		\item $n\neq 5$, $\eta$ and $\theta$ are both nonzero,
	\end{itemize}
	then $\cH_{k,s}(\eta, g)$ and $\cH_{k,t}(\theta, h)$ are not equivalent.
\end{theorem}

\begin{proof}
	Assume that $\cH_{k,s}(\eta, g)$ and $\cH_{k,t}(\theta, h)$ are equivalent. We use $\tau$ to denote one of the equivalence maps from $\cH_{k,s}(\eta, g)$ to $\cH_{k,t}(\theta, h)$. As in Lemma \ref{le:equivalence}, let $A$ be the universal support of $\{\tau(ax) : a\in\F_{q^n}\}$.
	
	According to the assumption, now $k=2$ or $k=n-2$. By Proposition \ref{prop:dual}, the Delsarte dual of $\cH_{n-2,s}(\eta, g)$ is equivalent to $\cH_{2,s}(-\eta, n-g)$.	As two MRD codes are equivalent if and only if their Delsarte duals are equivalent, we only have to consider the case in which $k=2$, i.e.\
	\begin{eqnarray*}
		\cH_{2,s}(\eta, g) &=& \{a_0 x + a_1 x^{q^s} + \eta a_0^{q^g} x^{q^{2s}}: a_0,a_1\in \F_{q^n} \},\\		
		\cH_{2,t}(\theta, h)&=& \{a_0 x + a_1 x^{q^t} + \theta a_0^{q^h} x^{q^{2t}}: a_0,a_1\in \F_{q^n} \}.
	\end{eqnarray*}
	
	Depending on the values of $\eta$ and $\theta$, we divide the proof into three cases.
	
	\vspace*{2mm}
	\noindent{\bf{Case 1.}}  If $\eta=\theta=0$, then we take $\cT:=\{\{0\},\{s\}\}$ which consists of independent supports of $\cH_{2,s}(0,g)$ and $\cS(\cH_{2,t}(0, h))= \{0,t\}$. It is not difficult to see that \eqref{eq:equivalence_lemma} holds for every $B\in \cT$ and $\cS(\cH_{2,t}(0, h))$ if and only if $s\equiv \pm t$. Of course, we can also get this result directly from Theorem \ref{th:main_inequivalence_1} $(a)$.
	
	\vspace*{2mm}
	\noindent{\bf{Case 2.}} If $\eta$ and $\theta$ are both nonzero, the universal support of $\cH_{2,t}(\theta, h)$ is $\cS := \{0,t,2t\}$ and we take $\cT :=\{\{0,2s\}, \{s\}\}$ consisting of the independent supports of $\cH_{2,s}(\eta, g)$.
	
	Assume that $s\not\equiv \pm t\pmod{n}$. We proceed to show the nonexistence of $A$ satisfying \eqref{eq:equivalence_lemma} for every $B\in \cT$ and $\cS$.
	
	First from $A^{\{s\}}\subseteq \cS$, we see that $A\subseteq\{-s,t-s,2t-s\}$. If $A=\{-s\}$, then
	\[A^{\{0,2s\}}=\{-s,s\},\]
	which cannot be a subset of $\cS=\{0,t,2t\}$. Similarly we can show that $A\neq \{t-s\}$ or $\{2t-s\}$.
	
	Assume that $A=\{t-s,2t-s\}$. Then
	\[
	A^{\{0,2s\}}= \left\{
	\begin{array}{ll}
	\{t-s,t+s,2t-s,2t+s\}, & \hbox{$t\not\equiv \pm 2s \pmod{n}$;} \\
	\{s,2t+s\}, & \hbox{$t\equiv 2s,t\not\equiv -2s \pmod{n}$;} \\
	\{t+s,2t-s\}, & \hbox{$t\equiv -2s,t\not\equiv 2s \pmod{n}$;} \\
	\emptyset, & \hbox{$t\equiv2s\equiv -2s\pmod{n}$.}
	\end{array}
		\right.	
	\]
	In the first case, $A^{\{0,2s\}}$ cannot belong to $\cS$. The second case is also not possible, because $t\notin\{ s, 2t+s\}$. The third case cannot hold, because $2t\notin \{t+s,2t-s\}$. The fourth case means $4s=n$, together with the assumption that $\gcd(n,s)=1$ we have $n=4$, $s=1$ and $t=2$, which contradicts the assumption that $\gcd(n,t)=1$. Similarly, we can verify that $A$ is neither $A=\{-s,t-s\}$ nor $A=\{-s,2t-s\}$.
	
	Finally, assume that $A=\{-s, t-s,2t-s\}$. Now we have
	\begin{eqnarray*}
		A^{\{0\}}&=&\{-s,t-s,2t-s\},\\
		A^{\{2s\}}&=&\{s,t+s,2t+s\}.
	\end{eqnarray*}
	As $t-s$ and $t+s$ cannot be in $\cS=\{0,t,2t\}$, from $A^{\{0,2s\}}\subseteq \cS$ we derive that $t\equiv 2s$ or $ -2s\pmod{n}$. If $t\equiv 2s$, we have $A^{\{0,2s\}}=\{-s,5s\}$; if $t\equiv -2s$, we have $A^{\{0,2s\}}=\{s,-5s\}$. As they  should both belong to $\cS=\{0,t,2t\}$, we get $n=5$ which has been excluded in our assumption.
	
	\vspace*{2mm}
	\noindent{\bf{Case 3.}} If $\eta\neq0$ and $\theta=0$, the universal support of $\cH_{2,t}(0, h)$ is $\cS := \{0,t\}$ and we take $\cT :=\{\{0,2s\}, \{s\}\}$ consisting of the independent supports of $\cH_{2,s}(\eta, g)$.
	
	Again we assume that $s\not\equiv \pm t\pmod{n}$ and show the nonexistence of $A$ satisfying \eqref{eq:equivalence_lemma} for every $B\in \cT$ and $\cS$.
	
	First from $A^{\{s\}}\in \cS$, we see that $A\subseteq\{-s,t-s\}$. If $A=\{-s\}$, then $A^{\{0,2s\}}=\{-s,s\}\not\subseteq\{0,t\}=\cS$. If $A=\{t-s\}$, then $A^{\{0,2s\}}=\{t-s,t+s\}\not\subseteq\{0,t\}=\cS$.
	
	If $A=\{-s,t-s\}$, then from
	\begin{eqnarray*}
		A^{\{0\}}&=&\{-s,t-s\},\\
		A^{\{2s\}}&=&\{s,t+s\},
	\end{eqnarray*}
	and $A^{\{0,2s\}}\subseteq \{0,t\}$, we see that $A^{\{0,2s\}}$ must be the empty set. It implies that $t\equiv 2s\equiv -2s \pmod{n}$. Hence $n=4$ and $t=2$, which contradicts the assumption that $\gcd(n,t)=1$.
\end{proof}

In Theorem \ref{th:main_inequivalence_2}, we see that the only open case is $n = 5$ with $\theta\neq 0$ and $\eta\neq 0$. In the next theorem, we concentrate on this case and present an answer to the equivalence problem.

\begin{theorem}\label{th:main_inequivalence_3}
	Let $k,s,t,g,h\in \Z^+$ satisfying $\gcd(5,s)=\gcd(5,t)=1$ and $k=2$ or $3$. Let $\eta$ and $\theta$ be in $\F_{q^5}^*$ satisfying $N_{q^{5s}/q^s}(\eta)\neq (-1)^{5k}$ and $N_{q^{5t}/q^t}(\theta)\neq (-1)^{5k}$. Assume that $s\not \equiv \pm t\pmod{5}$. Then $\cH_{k,s}(\eta, g)$ and $\cH_{k,t}(\theta, h)$ are not equivalent.
\end{theorem}
\begin{proof}
	Without loss of generality, we assume that $k=2$; for the case $k=3$, we consider the equivalence between their dual codes.
	
	Assume that $\cH_{2,s}(\eta, g)$ and $\cH_{2,t}(\theta, h)$ are equivalent. Let $\tau$ denote the equivalence map from $\cH_{2,s}(\eta, g)$ to $\cH_{2,t}(\theta, h)$, and let $A$ be the universal support of $\{\tau(ax) : a\in\F_{q^5}\}$.
	From the proof of Theorem \ref{th:main_inequivalence_2}, we know that $A=\{-s, t-s,2t-s\}$ and $t\equiv \pm 2s \pmod{5}$. Without loss of generality, we assume that $t\equiv 2s \pmod{5}$. Thus $A=\{-s, s, 3s\}$ and 
	\begin{eqnarray*}
		\cH_{2,s}(\eta, g) &=& \{a_0 x + a_1 x^{q^s} + \eta a_0^{q^g} x^{q^{2s}}: a_0,a_1\in \F_{q^5} \},\\		
		\cH_{2,t}(\theta, h)&=& \{b_0 x + b_1 x^{q^{2s}} + \theta b_0^{q^h} x^{q^{4s}}: b_0,b_1\in \F_{q^5} \}.
	\end{eqnarray*}
	
	As in the proof of Lemma \ref{le:monomial_equivalence}, we set 	$L_1(x)=\sum_{i=0}^{4}c_i x^{q^i}$ and $L_2(x)=\sum_{i=0}^{4}d_i x^{q^i}$, where $c_i$ and $d_i\in \F_{q^5}$. Then
	\[	\tau(ax)=L_1(a^\rho L_2(x))=\sum_{l=0}^{4}\left(\sum_{j=0}^{4} c_j(a^\rho d_{l-j})^{q^j}\right)x^{q^l}.\]
	Under $\tau$, the subset $\{a_1x^{q^s}: a_1\in \F_{q^5} \}\subseteq \cH_{2,s}(\eta, g)$ is mapped to
	\[\tau(a_1x^{q^s})=L_1(a_1^\rho (L_2(x))^{q^s})=\sum_{l=0}^{4}\left(\sum_{j=0}^{4} c_ja_1^{\rho{q^j}} d_{l-j-s}^{q^{j+s}}\right)x^{q^l}.\]
	As $\{\tau(a_1x^{q^s}): a_1\in \F_{q^5}\}\subseteq \cH_{2,t}(\theta, h)$, we have 
	\begin{eqnarray*}
		\sum_{j=0}^{4} c_j d_{-j}^{q^{j+s}} a_1^{\rho{q^j}}&= & 0,\\
		\sum_{j=0}^{4} c_j d_{2s-j}^{q^{j+s}} a_1^{\rho{q^j}}&=& 0,\\
		\theta \left(\sum_{j=0}^{4} c_ja_1^{\rho{q^j}} d_{-j-s}^{q^{j+s}}\right)^{q^h} &=& \sum_{j=0}^{4} c_ja_1^{\rho{q^j}} d_{3s-j}^{q^{j+s}}.
	\end{eqnarray*}
	for all $a_1\in \F_{q^5}$. Thus, for every $j\in \{0,1,\dots, 4\}$, we have
	\begin{eqnarray}
	\label{eq:n=5_cd_1}	c_j d_{-j} &=& 0,\\
	\label{eq:n=5_cd_2}	c_j d_{2s-j}&=& 0,\\
	\label{eq:n=5_cd_3}	\theta c_{j-h}^{q^h} d_{-j+h-s}^{q^{j+s}} &=& c_j d_{3s-j}^{q^{j+s}}.
	\end{eqnarray}
	
	Next we look at $\{\tau(a_0x + \eta a_0^{q^g}x^{q^{2s}}): a_0 \in \F_{q^5}\}$. First we compute
	\[\tau(a_0x) = L_1(a_0^\rho L_2(x)) = \sum_{l=0}^{4}\left(\sum_{j=0}^{4} c_j(a_0^\rho d_{l-j})^{q^j}\right)x^{q^l},\]
	and
	\[\tau(\eta a_0^{q^g}x^{q^{2s}}) = L_1((\eta a_0^{q^g})^\rho (L_2(x))^{q^{2s}}) = \sum_{l=0}^{4}\left(\sum_{j=0}^{4} c_j((\eta a_0^{q^g})^\rho d_{l-j-2s}^{q^{2s}})^{q^j}\right)x^{q^l}.\]
	Since the coefficients of $x^{q^s}$ and $x^{q^{3s}}$ of elements in $\cH_{2,t}(\theta, h)$ are always zero, we obtain that
	\[\sum_{j=0}^{4} c_j(a_0^\rho d_{l-j})^{q^j} + \sum_{j=0}^{4} c_j((\eta a_0^{q^g})^\rho d_{l-j-2s}^{q^{2s}})^{q^j} =0,\]
	for all $a_0\in \F_{q^5}$, $l=s$ and $3s$. It implies that the coefficients $c_j d_{l-j}^{q^j} + c_{j-g}\eta^{\rho q^{j-g}} d_{l+g-j-2s}^{q^{2s+j-g}}$ of $a_0^{\rho q^j}$ are all zero for $l=s$ and $3s$. In other words,
	\begin{equation}\label{eq:n=5_cdg_1}
		c_j d_{s-j}^{q^j} + c_{j-g}\eta^{\rho q^{j-g}} d_{g-j-s}^{q^{2s+j-g}}=0,
	\end{equation}
	and
	\begin{equation}\label{eq:n=5_cdg_2}
		c_j d_{3s-j}^{q^j} + c_{j-g}\eta^{\rho q^{j-g}} d_{s+g-j}^{q^{2s+j-g}}=0,
	\end{equation}	
	for all $j$. Furthermore, by looking at the coefficients of $x$ and $x^{q^{4s}}$ of $\tau(a_0x + \eta a_0^{q^g}x^{q^{2s}})$, we have
	\begin{align}\label{eq:n=5_cd_long}
			& \theta \left( \sum_{j=0}^{4} c_j(a_0^\rho d_{-j})^{q^j} + \sum_{j=0}^{4} c_j((\eta a_0^{q^g})^\rho d_{-j-2s}^{q^{2s}})^{q^j}\right)^{q^h}\\
	\nonumber= &\sum_{j=0}^{4} c_j(a_0^\rho d_{4s-j})^{q^j} + \sum_{j=0}^{4} c_j((\eta a_0^{q^g})^\rho d_{2s-j}^{q^{2s}})^{q^j}.
	\end{align}
	
	Next we consider several possible cases of $c_i$ and $d_i$ satisfying \eqref{eq:n=5_cd_1} and \eqref{eq:n=5_cd_2}. As we have proved that there are three elements in $A$, we only have to look at the following three cases.
	\begin{enumerate}[label=(\alph*)]
		\item Exactly two nonzero $c_i$ and two nonzero $d_i$.
		\item Exactly three nonzero $c_i$.
		\item Exactly three nonzero $d_i$.
	\end{enumerate}
	
	\vspace*{1.5mm}
	\noindent\textbf{Case (a):} Assume that $c_u$, $c_v$ and two $d_i$ are nonzero where $u\not\equiv v \pmod{5}$, and all the other $c_i$ and $d_i$ equal zero. By \eqref{eq:n=5_cd_1}, we have $d_{-u}=d_{-v}=0$. By \eqref{eq:n=5_cd_2}, we get $d_{2s-u}=d_{2s-v}=0$. As exactly two $d_i$ are nonzero and $u\not\equiv v \pmod{5}$, we must have
	\[-u \equiv 2s-v \pmod{5},\]
	or
	\[-v \equiv 2s-u \pmod{5}.\]
	Without loss of generality, we take $v=u-2s$. Then we have $c_u, c_{u-2s}$, $d_{-u+s}$ and $d_{-u+3s}$ are nonzero and all the other $c_i$ and $d_i$ are zero. Setting $j=u$ in \eqref{eq:n=5_cdg_1} and \eqref{eq:n=5_cdg_2}, we have
	\[c_u d_{s-u}^{q^u} + c_{u-g}\eta^{\rho q^{u-g}} d_{g-u-s}^{q^{2s+u-g}}=0,\]	
	and	
	\[c_u d_{3s-u}^{q^u} + c_{u-g}\eta^{\rho q^{u-g}} d_{s+g-u}^{q^{2s+u-g}}=0,\]
	from which we derive that $g$ must be congruent to $2s$ modulo $5$. Then we let $j=u-2s$ in \eqref{eq:n=5_cdg_1} and we have
	\[c_{u-2s} d_{3s-u}^{q^{u-2s}} + c_{u-4s}\eta^{\rho q^{u-4s}} d_{3s-u}^{q^{u-2s}}=0.\]
	However, as $c_{u-4s}=0$, from the above equation we derive that
	\[c_{u-2s} d_{3s-u}^{q^{u-2s}}=0,\]
	which is a contradiction to our assumption that $c_{u-2s}$ and $d_{3s-u}$ are nonzero.
	
	\vspace*{1.5mm}
	\noindent\textbf{Case (b):} Assume that there are exactly three nonzero $c_i$. Denote them by $c_u$, $c_v$ and $c_w$. By \eqref{eq:n=5_cd_1}, we get $d_{-u}=d_{-v}=d_{-w}=0$. By \eqref{eq:n=5_cd_2}, we get $d_{2s-u}=d_{2s-v}=d_{2s-w}=0$.  Without loss of generality, we assume that $v=u-2s$. By calculation, we see that there is at most one $d_i$ which is nonzero and we can take $w=u-4s$ without loss of generality. That means $c_u, c_{u-2s}$, $c_{u-4s}$ and $d_{-u+3s}$ are nonzero and all the other $c_i$ and $d_i$ equal zero. Letting $j=u$ in \eqref{eq:n=5_cdg_2}, we have
	\[c_u d_{3s-u}^{q^u} + c_{u-g}\eta^{\rho q^{u-g}} d_{s+g-u}^{q^{2s+u-g}}=0,\]
	from which we derive that $g$ must be congruent to $2s$ modulo $5$. By taking $j=u$ in \eqref{eq:n=5_cdg_2} and $j=u-2s$ in \eqref{eq:n=5_cdg_1}, we have
	\[c_u d_{3s-u}^{q^u} + c_{u-2s}\eta^{\rho q^{u-2s}} d_{3s-u}^{q^{u}}=0,\]
	and 
	\[c_{u-2s} d_{3s-u}^{q^{u-2s}} + c_{u-4s}\eta^{\rho q^{u-4s}} d_{3s-u}^{q^{u-2s}}=0.\]
	Hence,
	\begin{equation}\label{eq:n=5_c_u_1}
		c_u + c_{u-2s}\eta^{\rho q^{u-2s}}=0,
	\end{equation}
	and
	\begin{equation}\label{eq:n=5_c_u-2s}
		c_{u-2s}  + c_{u-4s}\eta^{\rho q^{u-4s}}=0.
	\end{equation}
	Moreover, from \eqref{eq:n=5_cd_3}, we derive that $h=4s$ and
	\begin{equation}\label{eq:n=5_c_u-4s}
		\theta c_{u-4s}^{q^{4s}}= c_u.
	\end{equation}
	Plugging $h=4s$, $g=2s$, the values of $c_i$ and $d_i$ into \eqref{eq:n=5_cd_long}, we have
	\[\theta c_u^{q^{4s}} \eta^{\rho q^{4s+u}} a_0^{\rho q^{u+s}} d_{3s-u}^{q^{u+s}} = c_{u-4s} a_0^{\rho q^{u-4s}}d_{3s-u}^{q^{u-4s}},\]
	for all $a_0$, i.e.\
	\begin{equation}\label{eq:n=5_c_u-4s_theta}
		\theta c_u^{q^{4s}} \eta^{\rho q^{4s+u}} = c_{u-4s}.
	\end{equation}
	From \eqref{eq:n=5_c_u_1}, \eqref{eq:n=5_c_u-2s} and \eqref{eq:n=5_c_u-4s_theta}, we have
	\begin{equation}\label{eq:n=5_case1_eta_1}
		c_u = \eta^{\rho (q^{u-2s} + q^{u-4s} + q^{u+4s})} \theta c_u^{q^{4s}}.
	\end{equation}
	From \eqref{eq:n=5_c_u_1}, \eqref{eq:n=5_c_u-2s} and \eqref{eq:n=5_c_u-4s}, we have
	\begin{equation}\label{eq:n=5_case1_eta_2}
		c_u = \theta c_u^{q^{4s}} \eta^{-\rho(q^{u+2s} + q^{u})}.
	\end{equation}
	Together with \eqref{eq:n=5_case1_eta_1}, we have
	\[\eta^{\rho q^u(1+q^s+q^{2s}+q^{3s}+q^{4s})}=1,\]
	which contradicts the assumption that $N_{q^{5s}/q^s}(\eta)\neq 1$.
	
	\vspace*{1.5mm}
	\noindent\textbf{Case (c):} According to the definitions of adjoint and equivalence of rank metric codes, two codes $\cC_1$ and $\cC_2$ are equivalent if and only if their adjoint codes are equivalent. It is clear that the adjoint of $L_1\circ f\circ L_2$ is $\hat{L}_2 \circ \hat{f} \circ \hat{L}_1$. It follows that the case of exactly three nonzero $d_i$ can be translated to the case of exactly three nonzero $c_i$ which has been already solved in case (b).
\end{proof}

\begin{proof}(Theorem \ref{th:equivalence_all})
	Combining the results of Theorems \ref{th:main_inequivalence_1}, \ref{th:main_inequivalence_2} and \ref{th:main_inequivalence_3}, we complete the proof of the necessity of $s\equiv \pm t \pmod{n}$.
	
	For $s\equiv \pm t \pmod{n}$, we proceed to prove the necessary and sufficient conditions for the equivalence between $\cH_{k,s}(\eta, g)$ and $\cH_{k,t}(\theta, h)$.
	
	We assume that $\cH_{k,s}(\eta, g)$ and $\cH_{k,t}(\theta, h)$ are equivalent. Again, we use $\tau$ to denote one of the equivalence maps from $\cH_{k,s}(\eta, g)$ to $\cH_{k,t}(\theta, h)$. As in Lemma \ref{le:equivalence}, let $A$ be the universal support of $\{\tau(ax) : a\in\F_{q^n}\}$.
	
	\vspace*{2mm}
	\noindent$(a)$ When $s\equiv t \pmod{n}$, $A$ must be equal to $\{0\}$. The reasons are as follows:
	\begin{enumerate}[label=(\roman*)]
		\item When $\eta=\theta=0$, $\cH_{k,s}(\eta,g)$ and $\cH_{k,t}(\theta,h)$ are both generalized Gabidulin codes $\cG_{k,s}$. Let $\cT(\cG_{k,s}):=\{\{is\}: i=0,\dots, k-1\}$, which is a collection of independent supports of $\cG_{k,s}$. The universal support of $\cG_{k,s}$ is $\cS=\{is : i=0,\dots,k-1\}$. It is straightforward to see that if \eqref{eq:equivalence_lemma} holds for $A$, all $B\in\cT(\cG_{k,s})$ and $\cS$, then $A$ must be $\{0\}$.
		\item When $\eta\neq 0$, from $A^{\{s\}}\subseteq \cS(\cH_{k,t}(\theta, h))$ we can derive that $A\subseteq \{-s,0,s\}$, which is equivalent to say that $A^{\{s\}}\subseteq \{0,s,2s\}$ by Lemma \ref{le:independent_support}. We can prove this result as follows: When $k=2$, there is nothing to prove. Assume that $k>2$ and $i_0s\in A^{\{s\}}$. Clearly $i_0\leqslant k$. By way of contradiction, we assume that $i_0>2$.  As $0<k+2-i_0<k$, we have $ax^{q^{s(k+2-i_0)}}\in \cH_{k,s}(\eta,g)$ for all $a\in \F_{q^n}$, i.e., $\{s(k+2-i_0)  \}$ is an independent support of $\cH_{k,s}(\eta,g)$. By Lemma \ref{le:equivalence}, $ A^{\{s(k+2-i_0)  \}}\subseteq \cS(\cH_{k,t}(\theta, h))$. However $i_0s\in A^{\{s\}}$ implies $(k+1)s\in A^{\{s(k+2-i_0)  \}}$ by Lemma \ref{le:independent_support}, which leads to a contradiction.
		
		Next let us show that $A=\{0\}$. First we look at the case in which $k>2$. Assume that $\tau(a_0x^{q^s})=g_0(a_0)x + g_1(a_0)x^{q^s} + g_2(a_0)x^{q^{2s}}$ for certain functions $g_0$, $g_1$ and $g_2$. Noting that $t=s$ and $\{s\}$, $\{2s\}$ are both independent supports of $\cH_{k,t}(\theta, h)$, i.e., $g_1(a_0)x^{q^s}$ and $g_2(a_0)x^{q^{2s}}\in \cH_{k,t}(\theta, h)$, we derive that $g_0(a_0)x\in \cH_{k,t}(\theta, h)$. It means $g_0(a_0)=0$ for each $a_0\in \F_{q^n}$. Hence $-s\notin A$.  Similarly one can show that $s\notin A$ using the assumption that $\eta\neq 0$.
		
		Now we look at the case $k=2$. If $-s\in A$, then under the map $\tau$,  $a_0x$ is mapped to $g_0(a_0)x^{q^{-s}} + g_1(a_0)x + g_2(a_0)x^{q^{s}}$ for certain functions $g_0$, $g_1$ and $g_2$, where $g_0$ is not identically zero. Meanwhile, $\eta a_0^{q^g}x^{q^{2s}}$ is mapped to $\tilde{g}_0(a_0)x^{q^{s}} + \tilde{g}_1(a_0)x^{q^{2s}} + \tilde{g}_2(a_0)x^{q^{3s}}$ for certain functions $\tilde{g}_0$, $\tilde{g}_1$ and $\tilde{g}_2$. 
		
		When $n>4$, we have $k+1=3<n-1$ which means that the term $g_0(a_0)x^{q^{-s}}$ is not constantly zero in $\tau(a_0x +\eta a_0^{q^g}x^{q^{2s}}) $. On the other hand, $\tau(a_0x +\eta a_0^{q^g}x^{q^{2s}}) \in \cH_{2,s}(\theta, h) $ for every $a_0\in \F_{q^n}$. However, there is no term $x^{q^{-s}}$ in any element of $\cH_{2,t}(\theta,h)$. Hence $-s\notin A$. Similarly we can show that $s\notin A$ using the assumption that $\eta\neq 0$. Therefore $A=\{0\}$.
		
		When $n=4$, the proof of $A=\{0\}$ is quite complicated and we put it in the Appendix.
	\end{enumerate}
	
	Now we know that $A=\{0\}$, which means that $\tau$ maps $\{a_0x: a_0\in \F_{q^n} \}$ to itself. By Lemma \ref{le:monomial_equivalence}, we know that $L_1=cx^{q^r}$ and $L_2=dx^{q^{n-r}}$ for certain $r$, $c$ and $d$. Thus the image of $\cH_{k,s}(\eta, g)$ is
	\[\{c a_0^{\rho q^r}d^{q^r}x + a_1 x^{q^s} + \dots + a_{k-1}x^{q^{s(k-1)}} + c\eta^{\rho q^r} a_0^{\rho q^{g+r}} d^{q^{sk+r}} x^{q^{sk}}: a_i\in \F_{q^n}    \}.\]
	It is the same as $\cH_{k,t}(\theta,h)$ if and only if
	\[\theta(c^{q^h-1}a_0^{\rho(q^{r+h}-q^{r+g})}d^{q^{r+h}- q^{sk+r}} )=\eta^{\rho q^r}, \]
	for all $a_0\in \F_{q^n}^*$. That means $h\equiv g \pmod{n}$ and
	\[\theta c^{q^h-1} d^{q^{r+h} - q^{r+sk}} = \eta^{\rho q^r}.\]
	
	\noindent$(b)$ When $s\equiv -t \pmod{n}$, we first apply $x\mapsto x^{q^{sk}}$ on $\cH_{k,t}(\theta, h)$ to get
	\[\{a_0^{q^{sk}}x^{q^{sk}} + a_1^{q^{sk}}x^{q^{s(k-1)}} + \dots + \theta^{q^{sk}}a_0^{q^{h+sk}}x: a_i\in \F_{q^n}   \}. \]
	It equals
	\[\{a_0x + a_1x^{q} + \dots + \theta^{-q^{sk}}a_0^{q^{-h}}x^{q^{sk}}: a_i\in \F_{q^n}   \}. \]
	Using the result for $s\equiv t\pmod{n}$, we have $h\equiv-g \pmod{n}$ and
	\[c^{q^{g}-1} d^{q^{r+g} - q^{r+sk}} = \eta^{\rho q^r}\theta^{q^{sk}}.\qedhere\]
\end{proof}

\begin{remark}
	Theorem \ref{th:equivalence_all} $(a)$ can also be directly used to completely determine the automorphism group of $\cH_{k,s}(\eta, g)$, which consists of $(L_1, L_2, \rho)$ where $L_1(x)=cx^{q^r}$, $L_2(x)=dx^{q^{n-r}}$ with $c,d\in \F_{q^n}^*$ and $\rho \in \Aut(\F_q)$ satisfying
	\[\eta c^{q^g-1} d^{q^{r+g} - q^{r+sk}} = \eta^{\rho q^r}.  \]
	In particular, when $s=1$ and $(n,k)\neq (4,2)$, this automorphism group has been calculated in \cite[Theorem 7]{sheekey_new_2016}, in which $\rho$ is ``decomposed" and its information is contained in $L_1(x):=\alpha x^{p^i}$ and $L_2(x):= \beta x^{p^{-i}}$, where $p=\mathrm{char}(\F_q)$.
\end{remark}
\begin{remark}
	In \cite{de_la_cruz_algebraic_2016}, the equivalence between MRD codes in $\F_q^{n\times n}$ is slightly different from ours. They use the isometries defined on $\F_q^{n\times n}$ by Wan in \cite{wan_geometry_1996} as the equivalence on MRD codes. In the language of linearized polynomials, besides the equivalence (Definition \ref{def:equivalence}) between $\cC'$ and $\cC$, we also have to check the equivalence between $\cC'$ and the adjoint code $\widehat{\cC}$ of $\cC$.
			
	However, even if we use the definition of equivalence on MRD codes in \cite{de_la_cruz_algebraic_2016}, by Theorem \ref{th:equivalence_all}, we can still determine the equivalence between different members of the generalized twisted Gabidulin codes.
\end{remark}

\section*{Appendix}
In this part, we prove that when $n=4$, $k=2$ and $s\equiv t\equiv 1 \pmod{4}$, for each equivalence map $\tau$ from $\cH_{2,1}(\eta, g)$ to $\cH_{2,1}(\theta, h)$, the universal support $A$ of $\{\tau(ax) : a\in\F_{q^4}\}$ equals $\{0\}$, i.e., $\tau$ maps monomials in $\F_{q^4}[x]$ to monomials. We need this result in the proof of Theorem \ref{th:equivalence_all}. It is worth pointing out that this is the only case of the equivalence of twisted Gabidulin codes which is not covered in Theorem 7 in \cite{sheekey_new_2016}.
\begin{proof}
	As in the proof of Theorem \ref{th:main_inequivalence_3}, we set $L_1(x)=\sum_{i=0}^{3}c_i x^{q^i}$ and $L_2(x)=\sum_{i=0}^{3}d_i x^{q^i}$, where $c_i$ and $d_i\in \F_{q^4}$. Then
	\[	\tau(ax)=L_1(a^\rho L_2(x))=\sum_{l=0}^{3}\left(\sum_{j=0}^{3} c_j(a^\rho d_{l-j})^{q^j}\right)x^{q^l}.\]
	Under $\tau$, the subset $\{a_1x^{q}: a_1\in \F_{q^4} \}\subseteq \cH_{2,1}(\eta, g)$ is mapped to
	\[\tau(a_1x^{q})=L_1(a_1^\rho (L_2(x))^{q})=\sum_{l=0}^{3}\left(\sum_{j=0}^{3} c_ja_1^{\rho{q^j}} d_{l-j-1}^{q^{j+1}}\right)x^{q^l}.\]
	As $\{\tau(a_1x^{q}): a_1\in \F_{q^4}\}\subseteq \cH_{2,1}(\theta, h)$, the coefficient of $x^{q^3}$ in $\tau(a_1x^{q})$ should be $0$, which means
	\[ \sum_{j=0}^{3} c_ja_1^{\rho q^j}d_{2-j}^{q^{j+1}}=0,  \]
	for all $a_1\in \F_{q^4}$. It implies that
	\begin{equation}\label{eq:n=4_equivalence_c_jd_{2-j}=0}
		c_jd_{2-j}=0,
	\end{equation}
	for $j=0,1,2,3$. From the coefficients of $x$ and $x^{q^2}$ in $\tau(a_1x^{q})\in  \cH_{2,1}(\theta, h)$, we also derive
	\[ \theta\left( \sum_{j=0}^{3}c_ja_1^{\rho q^j}d_{-j-1}^{q^{j+1}} \right)^{q^h} = \sum_{j=0}^3 c_ja_1^{\rho q^j}d_{1-j}^{q^{j+1}},\]
	for all $a_1\in \F_{q^4}$, which means
	\begin{equation}\label{eq:n=4_equivalence_2}
		\theta c_{j-h}^{q^h} d_{-j+h-1}^{q^{j+1}} = c_jd_{1-j}^{q^{j+1}},
	\end{equation}
	for $j=0,1,2,3$.
	
	We also need to look at $\tau(a_0 x + \eta a_0^{q^g}x^{q^2})$. By calculation,
	\[\tau(a_0 x + \eta a_0^{q^g}x^{q^2}) = 
		\sum_{l=0}^3\left( \sum_{j=0}^{3} c_j(a_0^\rho d_{l-j})^{q^j}\right)x^{q^l} + \sum_{l=0}^3\left( \sum_{j=0}^{3} c_j((\eta a_0^{q^g})^\rho d_{l-j-2}^{q^2})^{q^j}\right)x^{q^l}.
	\]
	As $\tau(a_0 x + \eta a_0^{q^g}x^{q^2})\in \cH_{2,1}(\theta, h)$, the coefficient of $x^{q^3}$, which is
	\[ \sum_{j=0}^{3} c_j(a_0^\rho d_{3-j})^{q^j} + c_j((\eta a_0^{q^g})^\rho d_{1-j}^{q^2})^{q^j},\]
	must equal $0$ for each $a_0\in \F_{q^4}$. It implies
	\begin{equation}\label{eq:n=4_equivalence_3}
		c_j d_{3-j}^{q^j} + c_{j-g}\eta^{(q^{j-g})\rho} d_{1-j+g}^{q^{2+j-g}}=0
	\end{equation}
	for $j=0,1,2,3$. Moreover, by comparing the coefficients of $x$ and $x^{q^2}$ in  $\tau(a_0 x + \eta a_0^{q^g}x^{q^2})$, we have
	\begin{equation*}
	\theta\left(  \sum_{j=0}^{3} c_j(a_0^\rho d_{-j})^{q^j}
			   +  c_j((\eta a_0^{q^g})^\rho d_{-j-2}^{q^2})^{q^j}
			  \right)^{q^h} = \sum_{j=0}^{3} c_j(a_0^\rho d_{2-j})^{q^j}
			  	   +  c_j((\eta a_0^{q^g})^\rho d_{-j}^{q^2})^{q^j}, 
	\end{equation*}
	for all $a_0\in\F_{q^4}$. By \eqref{eq:n=4_equivalence_c_jd_{2-j}=0}, it can be simplified into
	\begin{equation}\label{eq:n=4_equivalence_longest}
	\theta\left(  \sum_{j=0}^{3} c_j(a_0^\rho d_{-j})^{q^j}
			   			  \right)^{q^h} = \sum_{j=0}^{3} c_j((\eta a_0^{q^g})^\rho d_{-j}^{q^2})^{q^j}.
	\end{equation}
	
	By \eqref{eq:n=4_equivalence_c_jd_{2-j}=0}, it is easy to see that among $c_0$, $c_1$, $c_2$ and $c_3$ (resp.\ $d_0$, $d_1$, $d_2$ and $d_3$) there are at most 3 nonzero elements; otherwise all $d_i$'s (resp.\ $c_i$'s) are $0$ which means that $\tau$ is a zero map. It contradicts that $\tau$ is an equivalence map.
	
	Assume that $c_{j_0}\neq 0$. By \eqref{eq:n=4_equivalence_2},
	\begin{equation}\label{eq:n=4_equivalence_2_j_0}
		\theta c_{j_0-h}^{q^h} d_{-j_0+h-1}^{q^{j_0+1}} = c_{j_0}d_{1-j_0}^{q^{j_0+1}}.
	\end{equation}
	Depending on the values of the other $c_{i}$'s, we divide the rest of our proof into 3 cases.
	
	\vspace*{2mm}
	\noindent\textbf{Case (a):} Let us assume that $c_{j_0}$ is the unique nonzero element among the $c_i$'s. Our goal is to show that $d_{-j_0}$ is the unique nonzero one among the $d_i$'s, which means that $\tau$ maps monomials to monomials.
	
	When $h\not\equiv 0\pmod{4}$, from \eqref{eq:n=4_equivalence_2_j_0} we derive that $d_{1-j_0}=0$. From \eqref{eq:n=4_equivalence_3}, it is readily verified $d_{3-j_0}=0$ no matter what the value $g$ is. Together with $d_{2-j_0}=0$ (because of $c_{j_0}\neq 0$ and \eqref{eq:n=4_equivalence_c_jd_{2-j}=0}), we see that $d_{-j_0}$ is the unique nonzero one among the $d_i$'s.
	
	When $h\equiv 0 \pmod{4}$, \eqref{eq:n=4_equivalence_2_j_0} becomes
	\[ \theta c_{j_0} d_{-j_0-1}^{q^{j_0+1}} = c_{j_0}d_{1-j_0}^{q^{j_0+1}},\]
	which means
	\begin{equation}\label{eq:n=4_equivalence_2_j_0_h=0}
		\theta d_{3-j_0}^{q^{j_0+1}} = d_{1-j_0}^{q^{j_0+1}}.
	\end{equation}
	Now $d_{3-j_0}$ must be nonzero, which means $d_{1-j_0}\neq 0$ by \eqref{eq:n=4_equivalence_2_j_0_h=0}; otherwise all $d_i$'s are zero and we get a contradiction. Meanwhile, for $j=j_0$, \eqref{eq:n=4_equivalence_3} becomes
		\[		c_{j_0} d_{3-{j_0}}^{q^{j_0}} + c_{j_0-g}\eta^{(q^{j_0-g})\rho} d_{1-j_0+g}^{q^{2+j_0-g}}=0.\]
	Noting that the first term of the equation above is nonzero, we see that $g\equiv 0\pmod{4}$ and
	\begin{equation}\label{eq:n=4_equivalence_3_h=0} 
			d_{3-j_0}^{q^{j_0}} + \eta^{\rho q^{j_0}} d_{1-j_0}^{q^{2+j_0}}=0.
	\end{equation}
	Furthermore, by \eqref{eq:n=4_equivalence_longest}, we get
	\begin{equation}\label{eq:n=4_equivalence_3_one_ci}
		\theta d_{-j_0}^{q^{j_0}}  = d_{-j_0}^{q^{j_0+2}}\eta^\rho.
	\end{equation}
	From \eqref{eq:n=4_equivalence_2_j_0_h=0} and \eqref{eq:n=4_equivalence_3_h=0} we get
	\[ \frac{1}{\theta} = d_{1-j_0}^{q^{j_0+3}-q^{j_0+1}}(-\eta)^{\rho q^{j_0+1}}.\]
	Together with \eqref{eq:n=4_equivalence_3_one_ci}, we have
	\begin{equation}\label{eq:n=4_equivalence_one_ci_last_contra}
		1= \left(d_{1-j_0}^{q}d_{-j_0}\right)^{q^{j_0+2}-q^{j_0}}(-1) \eta^{\rho (q^{j_0+1}+1)}.
	\end{equation}
	Raise it to its $q^2$-th power, we have
	\[1= \left(d_{1-j_0}^{q}d_{-j_0}\right)^{q^{j_0}-q^{j_0+2}}(-1) \eta^{\rho (q^{j_0+3}+q^2)}.
	\]
	Together with \eqref{eq:n=4_equivalence_one_ci_last_contra}, we get
	\[1=\eta^{\rho (q^{j_0+1}+1+q^{j_0+3}+q^2)}.\]
	It follows that $N_{q^4/q}(\eta)=1$ no matter what value $j_0$ takes. It contradicts the assumption on the value of $\eta$.

	\vspace*{2mm}
	\noindent\textbf{Case (b):}  Let us assume that $c_{j_0+2}\neq 0$ and $c_{j_0\pm 1}=0$. It implies that $d_{2-j_0}=d_{-j_0}=0$. From \eqref{eq:n=4_equivalence_3}, we see that $g\equiv 0,2\pmod{4}$, otherwise all $d_i$'s are zero. 
	
	When $g\equiv 0\pmod{4}$, from \eqref{eq:n=4_equivalence_3} we obtain
	\begin{equation}\label{eq:n=4_equivalence_3_j_0}
		c_{j_0} d_{3-{j_0}}^{q^{j_0}} + c_{{j_0}}\eta^{\rho q^{{j_0}}} d_{1-{j_0}}^{q^{2+{j_0}}}=0,
	\end{equation}
	and
	\begin{equation}\label{eq:n=4_equivalence_3_j_0+2}
		c_{j_0+2} d_{1-{j_0}}^{q^{j_0+2}} + c_{{j_0+2}}\eta^{\rho q^{{j_0+2}}} d_{3-{j_0}}^{q^{{j_0}}}=0.
	\end{equation}
	Now $d_{1-j_0}$ and $d_{3-j_0}$ cannot be $0$, otherwise all $d_i$'s are $0$ which contradicts the definition of $\tau$.
	From \eqref{eq:n=4_equivalence_3_j_0} and \eqref{eq:n=4_equivalence_3_j_0+2} it follows that
	\[ \eta^{\rho(q^{j_0+2}+q^{j_0}) }=1.\]
	Taking $(q+1)$-th powers gives $N_{q^4/q}(\eta)^{\rho q^{j_0}}=1$, contradicting the assumption on $\eta$.
	
	When $g\equiv 2 \pmod{4}$, we can similarly derive that
	\[
			c_{j_0} d_{3-{j_0}}^{q^{j_0}} + c_{{j_0}+2}\eta^{\rho q^{{j_0+2}}} d_{3-{j_0}}^{q^{j_0}}=0,
	\]
	and 
	\[
			c_{j_0+2} d_{1-{j_0}}^{q^{j_0+2}} + c_{j_0}\eta^{\rho q^{{j_0}}} d_{1-{j_0}}^{q^{{j_0}+2}}=0.
	\]
	From them we can also obtain a contradiction to the assumption that $N_{q^4/q}(\eta)\neq 1$.
	
	\vspace*{2mm}
	\noindent\textbf{Case (c):}  Let us assume that $c_{j_0+1}\neq 0$. By \eqref{eq:n=4_equivalence_c_jd_{2-j}=0}, it implies that $d_{1-j_0}=0$. The proof for this case is the most complicated, and we have to deal with each of the $4$ possible values of $h$.
	
	If $h\equiv 0\pmod{4}$, then letting $j=j_0$ and $j_0+1$ respectively in \eqref{eq:n=4_equivalence_2}, we derive that $d_{3-j_0}=0$ and $d_{-j_0}=0$. As $d_{2-j_0}$ also equals $0$, it follows that $L_2(x)=0$ contradicting the definition of the equivalence map $\tau$.
	
	If $h\equiv 1\pmod{4}$, then letting $j=j_0+1$ in \eqref{eq:n=4_equivalence_2}, we have
	\[
		\theta c_{j_0}^{q} d_{-j_0-1}^{q^{j_0+2}} = c_{j_0+1}d_{-j_0}^{q^{j_0+2}},
	\]
	which implies $d_{3-j_0}$ and $d_{-j_0}$ are both nonzero; otherwise all $d_i$'s equal zero which contradicts that $\tau$ is an equivalence map. Moreover, \eqref{eq:n=4_equivalence_longest} becomes
	\[
	\theta ( c_{j_0}(a_0^\rho d_{-j_0})^{q^{j_0+1}}+c_{j_0+1}(a_0^\rho d_{3-j_0})^{q^{j_0+2}})
				   			   = c_{j_0}((\eta a_0^{q^g})^\rho d_{-j_0}^{q^2})^{q^{j_0}}+c_{j_0+1}((\eta a_0^{q^g})^\rho d_{3-j_0}^{q^2})^{q^{j_0+1}}.
	\]
	As the equation above holds for all $a_0$, there must be $g\equiv 1\pmod{4}$. Plugging it and $j=j_0$ into \eqref{eq:n=4_equivalence_3}, we get
	\[c_{j_0}d_{3-j_0}^{q^{j_0}}=0,\]
	because $c_{j_0-1}=0$. It contradicts that $c_{j_0}\neq 0$ and $d_{3-j_0}\neq 0$.
	
	If $h\equiv 2 \pmod{4}$,  then letting $j=j_0+1$ and $j_0+2$ respectively in \eqref{eq:n=4_equivalence_2}, we have
	\begin{equation}\label{eq:n=4_equivalence_case_c_h=2}
		\theta c_{j_0+3}^{q^2} d_{-j_0}^{q^{j_0+2}} = c_{j_0+1}d_{-j_0}^{q^{j_0+2}},
	\end{equation}
	and
	\[
		\theta c_{j_0}^{q^2} d_{3-j_0}^{q^{j_0+3}} = c_{j_0+2}d_{3-j_0}^{q^{j_0+3}}.
	\]
	From these two equations above and \eqref{eq:n=4_equivalence_c_jd_{2-j}=0}, we can derive that there is exactly one nonzero element in $\{d_{-j_0}, d_{3-j_0}\}$. 
	
	We only consider the case in which $d_{-j_0}\neq 0$; the other case can be handled in an analogous way. Now by \eqref{eq:n=4_equivalence_c_jd_{2-j}=0} and \eqref{eq:n=4_equivalence_case_c_h=2}, $c_{j_0+2}=0$ and $c_{j_0+3}\neq 0$. By taking $j=j_0+3$ in \eqref{eq:n=4_equivalence_3}, we have
	\begin{equation}\label{eq:n=4_equivalence_case_c_h=2_longest}
		c_{j_0+3} d_{-j_0}^{q^{j_0+3}} + c_{j_0+3-g}\eta^{(q^{j_0+3-g})\rho} d_{2-j_0+g}^{q^{1+j_0-g}}=0.
	\end{equation}
	As the first term of the equation above is nonzero and $d_{-j_0}$ is the unique nonzero element among the $d_i$'s, there must be $g\equiv 2 \pmod{4}$ and \eqref{eq:n=4_equivalence_case_c_h=2_longest} becomes
	\[		c_{j_0+3} d_{-j_0}^{q^{j_0+3}} + c_{j_0+1}\eta^{\rho q^{j_0+1}} d_{-j_0}^{q^{j_0+3}}=0,\]
	which means
	\[c_{j_0+3}  + c_{j_0+1}\eta^{\rho q^{j_0+1}} =0.\]
	Together with \eqref{eq:n=4_equivalence_case_c_h=2}, we obtain
	\begin{equation}\label{eq:n=4_equivalence_case_c_h=2_last}
		\theta c_{j_0+3}^{q^2-1} = -1/\eta^{\rho {q^{j_0+1}}}.
	\end{equation}
	On the other hand, \eqref{eq:n=4_equivalence_longest} now becomes
	\[	\theta ( c_{j_0}(a_0^\rho d_{-j_0})^{q^{j_0}})^{q^2}
		= c_{j_0}((\eta a_0^{q^2})^\rho d_{-j_0}^{q^2})^{q^{j_0}},\]
	which implies
	\[ \theta c_{j_0}^{q^2} = c_{j_0}\eta^{\rho q^{j_0}} .\]
	Together with \eqref{eq:n=4_equivalence_case_c_h=2_last}, we have
	\[ \left(\frac{c_{j_0+3}}{c_{j_0}}\right)^{q^2-1} = \frac{-1}{(\eta^{\rho q^{j_0}} )^{q+1}},
	\]
	which means that 
	\[N_{q^4/q}\left(\frac{1}{\eta^{\rho q^{j_0}} }
	\right) =  \left(\frac{c_{j_0+3}}{c_{j_0}}\right)^{(q^2-1)(q^2+1)}=1.\]
	It contradicts with the assumption $N_{q^4/q}(\eta)\neq 1$.

	If $h\equiv 3\pmod{4}$, then letting $j=j_0+1$   in \eqref{eq:n=4_equivalence_2}, we have
	\[
		\theta c_{j_0+2}^{q^3} d_{1-j_0}^{q^{j_0+2}} = c_{j_0+1}d_{-j_0}^{q^{j_0+2}}.
	\]	
	Together with $d_{1-j_0}=0$ and $c_{j_0+1}\neq 0$, we derive that $d_{-j_0}=0$. Similarly by letting $j=j_0+3$ in \eqref{eq:n=4_equivalence_2}, we can show that $d_{2-j_0}=0$. It means that $\tau$ is a zero map which contradicts the definition of an equivalence map.
	
	\vspace*{2mm}
	\noindent\textbf{Case (d):}  Let us assume that $c_{j_0-1}\neq 0$. If we replace $j_0$ by $j_0+1$, it turns out to be Case (c), which has been already proved.
\end{proof}

\section*{acknowledgment}
This work is supported by the Research Project of MIUR (Italian Office for University and Research) ``Strutture geometriche, Combinatoria e loro Applicazioni" 2012. Yue Zhou is partially supported by the National Natural Science Foundation of China (No.\ 11771451).

\end{document}